\documentclass[12 pt, reqno]{amsart}

\usepackage{fullpage}

\usepackage{amsmath}
\usepackage{amsfonts}
\usepackage{amssymb}
\usepackage{amsthm}
\usepackage{comment}
\usepackage{epsfig}
\usepackage{psfrag}
\usepackage{mathrsfs}
\usepackage{amscd}
\usepackage[all]{xy}
\usepackage{rotating}
\usepackage{lscape}
\usepackage{amsbsy}

\usepackage{color}
\usepackage{bbm}
\usepackage{eucal}
\usepackage{tikz}
\usepackage{tikz-cd}
\usetikzlibrary{decorations.markings,arrows}
\usetikzlibrary{patterns}
\usetikzlibrary{cd}
\usepackage{hyperref}

\usepackage{chngcntr}
\counterwithin{figure}{section}
\newtheorem{thm}{Theorem}[section] 
\newtheorem{theorem}[thm]{Theorem}
\newtheorem{lemma}[thm]{Lemma}
\newtheorem{corollary}[thm]{Corollary}
\newtheorem{definition}[thm]{Definition}
\newtheorem{prop}[thm]{Proposition} 
\theoremstyle{plain}
\newtheorem{example}[thm]{Example}
\newtheorem*{example*}{Example}

\DeclareMathOperator{\Hom}{\sf Hom}
\DeclareMathOperator{\m}{\sf Mod}

\def\cA{\mathcal A}\def\cB{\mathcal B}\def\cC{\mathcal C}
\def\cH{\mathcal H}

\def\cQ{\mathcal Q}

\def\cZ{\mathcal Z}

\def\RR{\mathbb R}

\def\ZZ{\mathbb Z}

\DeclareMathOperator{\id}{{\sf id}}

\newcommand{\Ring}{\mathsf{Ring}}
\def\bb1{\mathbb 1}
\DeclareMathOperator{\Mon}{\sf Mon}

\begin{document}

\title{Lax Functors, Cospans, and the Center Construction}

\author{Ryan Grady
\and Garrett Oren}
% \thanks{Department of Mathematical Sciences, Montana State University
%\newline
%{\scriptsize \tt{ ryan.grady1@montana.edu, annaschenfisch@montana.edu}}}

\address{Department of Mathematical Sciences\\Montana State University\\Bozeman, MT 59717}
\email{ryan.grady1@montana.edu}

\address{Department of Mathematical Sciences\\Montana State University\\Bozeman, MT 59717}
\curraddr{Mathematics \& Statistics\\Boston University\\Boston, MA 02215}
\email{goren@bu.edu}

\begin{abstract}
The center construction is not (classically) functorial. In this note, we specialize a universal construction of Jacob Lurie to the category of rings and upgrade the classical center to a lax functor.  In particular, we find  lax functors to the \textit{Morita bicategory} and  the \textit{bicategory of cospans}. 
\end{abstract}

\keywords{Centralizer, Functoriality, Morita Category, Cospans}
\subjclass[2010]{Primary 18D10. Secondary 16D20, 18B10, 18D05. }

\maketitle

\tableofcontents

%%%%%%%%%%%%%%%%%%%%%%%
\section{Introduction}

Functoriality (compatibility with composition) is a central notion in modern mathematics.  Yet, as we recall in a bit, passing from a ring to its center does not define a functor from the category of rings to itself. In this note, we present a streamlined approach to upgrading the center construction to a functor. The key ideas are:
\begin{enumerate}
\item Define the center of a ring in terms of a \emph{universal property};
\item Utilize \emph{bicategories} to relax the notion of functoriality. Functoriality (or the correction thereof) will then be encoded by additional coherence data.
\item Change the codomain category to the \emph{Morita bicategory} or the \emph{bicategory of cospans};

\end{enumerate}

None of the technical mathematics in this note is original, though we hope our presentation is novel and illuminating. The description of the center construction in terms of a universal property is drawn from Lurie's \emph{Higher Algebra} \cite{Lurie}; Lurie works in the language of stable $\infty$-categories while our description is 1-categorical. Our description of functoriality is 2-categorical and we choose to work in bicategories for this purpose.

The question of functoriality of the center construction has also been extensively studied from the perspective of Rational Conformal Field Theory and Topological Field Theory over the last 20+ years. The results we present could also be extracted from this literature by a careful reader, though again we hope that our presentation is more classical and palatable to a more junior researcher. For further reading on the relationship to field theory, we suggest the work of Davydov and collaborators \cite{Davydov}, \cite{DKR1}, \cite{DKR2}.

\begin{example*}
To see the lack of functoriality of the classical center for rings, we will bootstrap an example from groups. Let 
\[
D_6 = \langle s,r | s^2 = r^3 = e, srs=r^{-1} \rangle
\]
be the dihedral group of order 6. Consider $\ZZ/2$ written additively and the group homomorphisms
\[
\varphi \colon \ZZ/2 \to D_6; \quad 1 \mapsto s, \quad \text{ and } \quad \psi \colon D_6 \to \ZZ/2; \quad s \mapsto 1, \; r \mapsto 0.
\]
We then have the $\ZZ$-linear extensions of the maps at the level of integral group rings
\[
\widetilde{\varphi} \colon \ZZ[\ZZ/2] \to \ZZ[D_6] \quad \text{ and } \quad \widetilde{\psi} \colon \ZZ[D_6] \to \ZZ[\ZZ/2] .
\]
Note that by construction $\widetilde{\psi} \circ \widetilde {\varphi} \colon \ZZ[\ZZ/2] \to \ZZ[\ZZ/2]$ is an isomorphism. Now, $\ZZ[\ZZ/2]$ is commutative, while the center of $\ZZ[D_6] \cong \ZZ$. Hence, the induced map
\[
\widetilde{\psi}_\ast \colon Z(\ZZ[D_6])\cong \ZZ \to \ZZ[\ZZ/2] \cong Z(\ZZ[\ZZ/2])
\]
is not surjective, so neither is the composition $\widetilde{\psi}_\ast \circ \widetilde{\varphi}_\ast \colon \ZZ[\ZZ/2] \to \ZZ[\ZZ/2]$. However, as we have already observed, $\left ( \widetilde{\psi} \circ \widetilde{\varphi} \right )_\ast \colon \ZZ[\ZZ/2] \to \ZZ[\ZZ/2]$ is an isomorphism!
\end{example*}

In the present note, we begin by recalling some basics of monoidal categories, bicategories, and functors between them. In Section \ref{cent} we define centers and centralizers via a universal construction. Section \ref{func1} contains the proof that the center defines a lax functor to the Morita bicategory. In Section \ref{func2} we recall the bicategory of cospans---which are a diagrammatic generalization of bimodules---and prove that the center defines a lax functor from rings to cospans. Section \ref{tft} concludes by outlining the relevance of the centralizer construction to topological field theories.

\subsection*{Acknowledgements} The authors are grateful to David Ayala for his feedback and suggestions. We thank Ingo Runkel for comments on a previous version and for explicit comparison to \cite{DKR1}. Finally, the authors thank the anonymous referee for useful feedback which has greatly improved the readability of the article.

%%%%%%%%%%%%%%%%%%%%%%%
\section{Preliminaries}

We now provide (a terse) recollection of some categorical notions we will need in the present article. A classic reference for this material is Maclane \cite{Maclane}; the more recent \cite{Riehl} is also excellent.

\begin{definition} A \emph{monoidal category} is consists of a category $\cC$ with
\begin{itemize}
\item A bifunctor $\otimes \colon \cC \times \cC \to \cC$;
\item A distinguished object, $\mathbbm{1} \in \cC$, the ``unit"; and 
\item Three (families of) natural isomorphisms:
\begin{itemize}
\item For each triple of objects $A,B,C \in \cC$, an ``associator" with components 
\[
\alpha_{A,B,C} \colon A \otimes (B \otimes C) \cong (A \otimes B) \otimes C ;
\] 
\item Left and right ``unitors" with components
\[
\lambda_A \colon \mathbbm{1} \otimes A \cong A \quad \text{ and } \quad \rho_A \colon A \otimes \mathbbm{1} \cong A,
\]
for each object $A \in \cC$.
\end{itemize}
\end{itemize}
The associator and unitors satsify certain coherence conditions.
\begin{itemize}
\item For each triple, the following diagram commutes.
\begin{center}
\begin{tikzcd}
A \otimes (\mathbbm{1} \otimes B) \arrow[rd,"\id_A \otimes \lambda_B"'] \arrow[rr,"\alpha_{A,\mathbbm{1},B}"] & & (A\otimes \mathbbm{1}) \otimes B \arrow[ld,"\rho_A \otimes \id_B"]\\ 
& A \otimes B
\end{tikzcd}
\end{center}
\item For each quadruple of objects, the components of the associator satisfies the commutative ``pentagon diagram."
\begin{small}
\begin{center}
\begin{tikzcd}[column sep=-2em]
&&(A \otimes B) \otimes (C \otimes D) \arrow[rrdd,"\alpha_{A\otimes B,C,D}"] &&\\
&\\
A \otimes (B \otimes (C \otimes D)) \arrow[rruu,"\alpha_{A,B,C\otimes D}"] \arrow[rdd,"\id_A \otimes \alpha_{B,C,D}"'] &&&& ((A \otimes B)\otimes C) \otimes D \\
&\\
& A \otimes ((B \otimes C) \otimes D) \arrow[rr,"\alpha_{A,B\otimes C, D}"'] && (A \otimes (B \otimes C)) \otimes D \arrow[ruu,"\alpha_{A,B,C} \otimes \id_D"']
\end{tikzcd}
\end{center}
\end{small}
\end{itemize}
\end{definition}

Given a monoidal category $(\cC, \otimes, \mathbbm{1})$, the monoid objects in $\cC$ themselves form a category, where  a monoid object is an object $M \in \cC$ equipped with a multiplication morphism $\mu \colon M \otimes M \to M$ and a unit morphism $\eta \colon \mathbbm{1} \to M$ such that $\mu$ is associative (up to the associator $\alpha$) and unital, i.e., 

\begin{center}
\begin{tikzcd}
\mathbbm{1} \otimes M \arrow[rd,"\lambda_M"'] \arrow[r,"\eta \otimes \id_M"] & M \otimes M \arrow[d,"\mu"] & M \otimes \mathbbm{1} \arrow[l,"\id_M \otimes \eta"'] \arrow[ld,"\rho_M"] \\
&M
\end{tikzcd}
\end{center}
is a commutative diagram.

Moreover, a monoidal category is \emph{symmetric} if for each pair of objects $A,B \in \cC$ there is an isomorphism $\tau_{A,B} \colon A \otimes B \cong B \otimes A$ which is natural in both $A$ and $B$ and which satisfies three additional coherence axioms:
\begin{itemize}
\item For each pair $A,B \in \cC$, $\tau_{A,B} \circ \tau_{B,A} = \id_{A \otimes B}$.
\item The diagram
\begin{center}
\begin{tikzcd}
A \otimes \mathbbm{1} \ar[d,swap,"\rho_A"] \ar[r,"\tau_{A,\mathbbm{1}}"] & \mathbbm{1} \otimes A \ar[d,"\lambda_A"] \\
A \ar[r,"\id_A"] & A
\end{tikzcd}
\end{center}
commutes for each object $A \in \cC$.
\item (The Hexagon Axiom.) For all objects $A,B,C \in \cC$ the following commutes.
\begin{center}
\begin{tikzcd}
&A \otimes (C \otimes B) \ar[rr,"\id_A \otimes \tau_{Z,Y}"]  &&A \otimes (B \otimes C) \ar[dr,"\alpha_{A,B,C}"]\\
(A \otimes C) \otimes B \ar[ur,"\alpha_{A,C,B}^{-1}"] \ar[dr,swap,"\tau_{A\otimes C,B}"]&&&&(A \otimes B) \otimes C \\
& B \otimes (A \otimes C) \ar[rr,"\alpha_{B,A,C}"]&& (B \otimes A) \otimes C \ar[ur,swap,"\tau_{B,A} \otimes \id_C"]
\end{tikzcd}
\end{center}
\end{itemize}

\begin{definition} A \emph{bicategory}, $\cB$, consists of the following data:
\begin{itemize}
\item A collection of objects, $\mathrm{obj}(\cB)$; 
\item For each pair of objects $X,Y \in \mathrm{obj}(\cB)$ a category $\cB(X,Y)$, whose objects are called 1-morphisms and whose morphisms are called 2-morphisms;
\item For each object $X$, a distinguished 1-morphism $\id_X \in \cB(X,X)$;
\item For each triple of objects $X,Y,Z$, a functor, ``composition," $\circ \colon \cB(Y,Z) \times \cB(X,Y) \to \cB(X,Z)$; 
\item For each triple of composable 1-morphisms a natural isomorphism 
\[
\alpha \colon (h \circ g) \circ f \overset{\sim \;}{\Rightarrow} h \circ (g \circ f);
\] 
\item For each pair of objects and 1-morphism $f \in \cB(X,Y)$,  natural (in $f$) isomorphisms
\[
r_{X,Y} \colon f \circ \id_X \overset{\sim \;}{\Rightarrow} f \quad \text{ and } \quad \ell_{X,Y} \colon \id_Y \circ f \overset{\sim \;}{\Rightarrow} f.
\]
\end{itemize}
The associator must satisfy a pentagon diagram as before. Similarly, the unitors $r$ and $\ell$ must satisfy the triangular diagrams from above.
\end{definition}

Note that any category can be promoted to a bicategory. Indeed, if $\cC$ is an ordinary category then we can consider it as a bicategory where for any pair of objects $X,Y \in \cC$, the category $\cC (X,Y)$ is discrete, i.e., it only has identity morphisms. For more on bicategories, see Chapters 2 and 4 of \cite{JY}.

In the sequel bicategories will arise as they are a setting which allows for more flexible notions of functoriality.

\begin{definition}\label{def:lax}
Let $\cA$ and $\cB$ be bicategories. A \emph{lax functor} $P \colon \cA \to \cB$ consists of
\begin{itemize}
\item A function $P \colon \mathrm{Obj}(\cA) \to  \mathrm{Obj}(\cB)$;
\item For each hom-category $\cA(X,Y)$ in $\cA$, a functor 
\[
P_{X,Y} \colon \cA(X,Y)\to \cB(P(X),P(Y));
\]
\item For each object $X\in \cA$ a 2-cell $P_{\id_X} \colon \id_{P(X)} \Rightarrow P_{X,X}(\id_X)$;
\item For each triple of objects and morphisms $f \colon X \to Y$ and $g \colon Y \to Z$ , a natural (in $f$ and $g$) transformation 
\[
P_{g,f} \colon P_{Y,Z} (g) \circ P_{X,Y}(f)  \Rightarrow   P_{X,Z}(g \circ f).
\] 
\end{itemize}
This data satisfies a sequence of coherence diagrams specifying unity and associativity. Let $f \in \cA(W,X)$, $g \in \cA (X,Y)$, and $h \in \cA (Y,Z)$. 
\begin{itemize}
\item (Lax Unity.) The following commute in $\cB (P(W), P(X))$.
\begin{center}
\begin{tiny}
\begin{tikzcd}
\id_{P(X)} \circ P_{W,X} (f) \ar[d,swap,"P_{\id_X} \ast \id_{P_{W,Z} (f)}"] \ar[r,"\ell_\cB"] & P_{W,X}(f) \\
P(\id_X) \circ P_{W,X} (f) \ar[r,"P_{\id_X, f}"] & P_{W,X} (\id_X \circ f) \ar[u,swap,"P_{W,X}(\ell_\cA)"]
\end{tikzcd}
\hspace{0.15in}
\begin{tikzcd}
P_{W,X} (f) \circ \id_{F(W)} \ar[r,"r_\cB"] \ar[d,swap,"\id_{P_{W,X}(f)} \ast P_{\id_X}"] & P_{W,X}(f) \\ P_{W,X}(f) \circ P_{W,W} (\id_W) \ar[r,"P_{f,\id_W}"] &P_{W,X} (f \circ \id_W) \ar[u,swap,"P_{W,X}(r_\cA)"]
\end{tikzcd}
\end{tiny}
\end{center}
\item (Lax Associativity.) The following commutes in $\cB (P(W), P(Z))$.
\begin{center}
\begin{tikzcd}
(P_{Y,Z} (h) \circ P_{X,Y} (g)) \circ P_{W,X}(f) \ar[r,"\alpha_\cB"] \ar[d,swap,"P_{h,g} \ast \id_{P_{W,Z}(f)}"] & P_{Y,Z}(h) \circ (P_{X,Y} (g) \circ P_{W,X} (f)) \ar[d,"\id_{P_{Y,Z} (h)} \ast P_{g,f}"] \\
P_{X,Z} (h \circ g) \circ P_{W,X} (f) \ar[d,swap,"P_{h\circ g,f}"] & P_{Y,Z} (h) \circ P_{W,Y} (g \circ f) \ar[d,"P_{h,g \circ f}"]\\
P_{W,Z} ((h \circ g) \circ f) \ar[r,"P_{W,Z} (\alpha_\cA)"] & P_{W,Z} (h \circ (g \circ f))
\end{tikzcd}
\end{center}
\end{itemize}
\end{definition}

%A \emph{psuedo-functor} is a lax functor $P\colon \cA \to \cB$ such that all $P_{\id_X}$ are invertible and such that each transformation $P_{f,g}$ is a natural isomorphism.
%
%As the domain and codomain are bicategories, we have further flexibility in variance, in particular we will be interested in the case where the structural 2-morphisms and natural transformations go the other way around.
%
%\begin{definition}
%Let $\cA$ and $\cB$ be bicategories. An \emph{oplax functor} $L \colon \cA \to \cB$ is a lax functor $L^{op} \colon \cA^{op} \to \cB^{op}$.
%\end{definition}

%%%%%%%%%%%%%%%%%%%%%%%
\section{Centralizers}\label{cent}

Let $(\cC, \otimes, \mathbbm{1})$ be a symmetric monoidal category and $\Mon (\cC)$ the associated category of monoids.

\begin{definition}
The \emph{category of factorizations} for some category of monoids $\Mon(\cC)$ and a morphism $f\in \Hom_{\Mon(\cC)}(A,B)$ denoted $\mathsf{Factz}(\Mon (\cC),f)$ has as objects pairs $(C,g)$ where $C\in \Mon(\cC)$ and $g \colon A\otimes C \to B$ is a map of monoids and the following diagram commutes 

\begin{center}
\begin{tikzcd}[sep=large]
  & A\otimes C \arrow[rd, "g"]& \\
  A\arrow[ur, " \eta_c"]\arrow[rr,"f"]& & B,
  \end{tikzcd}
\end{center}
where $\eta_c:=(\id_A \otimes \eta) \circ \rho_{A}^{-1} $ for $\eta$ is the unit map for the monoid structure on $C$ and $\rho_{A}^{-1}$ is the inverse of the right unitor. 
 For two objects $(C,g),(D,h)$ then $k\in \Hom_{\mathsf{Factz}}((C,g),(D,h))$ is a morphism $k\in \Hom_{\Mon(\cC)}(C,D)$ such that the following diagram commutes.

\begin{center}
\begin{tikzcd}[sep=large]
 & A\otimes C \arrow[d,"\id_A \otimes k"] \arrow[rdd, bend left, "g"] & \\
 & A \otimes D \arrow[rd, "h"] & \\
 A \arrow[uur,bend left, "\eta_C"] \arrow[ur,"\eta_D"] \arrow [rr, "f"]& & B
\end{tikzcd}
\end{center}
Composition of morphisms in $\mathsf{Factz}(\Mon(\cC),f)$ is just regular composition in $\Mon(\cC)$. The identity morphism $\id_c\in \Hom((C,g),(C,g))$ is just $\id_C$ since we always have the  diagram

\begin{center}
\begin{tikzcd}[sep=large]
 & A\otimes C \arrow[d,"\id_A \otimes \id_C"] \arrow[rdd, bend left, "g"] & \\
 & A \otimes C \arrow[rd, "g"] & \\
 A \arrow[uur,bend left, "\eta_C"] \arrow[ur,"\eta_C"] \arrow [rr, "f"]& & B.
\end{tikzcd}
\end{center}
\end{definition}

\begin{definition}
Let $f \colon A \to B$ be a morphism in $\Mon (\cC)$, then the centralizer of $f$ is a final object in the category $\mathsf{Factz}(\Mon(\cC),f)$. For $A\in \Mon(\cC)$, the \emph{center} of $A$ is the centralizer of the identity morphism of $A$.
\end{definition}

As it satisfies a universal property, the centralizer is well-defined up to isomorphism (should it exist), so hence we refer to it as ``the" centralizer. (This is an important, yet standard, argument found in any of the references of the previous section \cite{JY}, \cite{Maclane}, \cite{Riehl}.)

As a consistency check, we verify that the preceding definition recovers the classical definition of the center of a ring. Let $\cC=\mathsf{Ab}$, the category of Abelian groups, so $\Mon (\cC)$ is identified with the category of unital rings, $\mathsf{Ring}$.

\begin{prop}\label{203}
If $f \colon R \to S$ is a ring homomorphism, then $Z(f)$ is given by
\[
Z(f)= \{s\in S | s f(r)=f(r) s \quad \forall r\in R\}
\] 
with the morphism $\mu_f \in \Hom_{\Ring} ( R\otimes Z(f), S)$ defined as follows 
\[
\mu_f=\mu \circ (f\otimes \id),
\] where $\mu$ is just multiplication in $S$.

\end{prop}
\begin{proof}
With the definitions given above,  the  diagram 
\begin{center}
\begin{tikzcd}[sep=large]
  & R\otimes Z(f) \arrow[rd, "\mu_f"]& \\
  R\arrow[ur, " \eta"]\arrow[rr,"f"]& & S
  \end{tikzcd}
\end{center}
  commutes. Now to show that it satisfies the universal property we will let $(T,\tau)$ be any other object in $\mathsf{Factz} (\Ring,f)$, so a commutative diagram
\begin{center}
\begin{tikzcd}[sep=large]
  & R\otimes T \arrow[rd, "\tau"]& \\
  R\arrow[ur, " \eta"]\arrow[rr,"f"]& & S.
  \end{tikzcd}
\end{center}
We need a morphism $\tilde{\tau}\in \Hom_{\Ring}(T, Z(f))$ such that 
\begin{center}
\begin{tikzcd}[sep=large]
 & R\otimes T \arrow[d,"\id_A \otimes \tilde{\tau}"] \arrow[rdd, bend left, "\tau"] & \\
 & R \otimes Z(f) \arrow[rd, "\mu_f"] & \\
 R \arrow[uur,bend left, "\eta"] \arrow[ur,"\eta"] \arrow [rr, "f"]& & S
\end{tikzcd}
\end{center}
commutes. The morphism $\tilde{\tau}$ can be defined as  
\[
\tilde{\tau}(t)=\tau(1\otimes t).
\] 

We will now show that this morphism is well defined. We know by the properties of tensor products that 
\[
\tau(r\otimes t)=\tau(r\otimes 1)\tau(1\otimes t)=\tau(1\otimes t)\tau(r\otimes 1).
\]
By commutativity we know that 
\[
f(r)\tau(1\otimes t)= \tau(1\otimes t) f(r).
\]
Therefore, $\tau(1\otimes t)\in Z(f)$ so we have a well defined map. It is also a morphism in $\Ring$ since it respects addition and multiplication by construction and 
\[
\tilde{\tau}(1)=\tau(1\otimes 1)=1 \qquad \tilde{\tau}(0)=\tau(1\otimes 0)=0.
\]
This map also makes the diagram commute. The upper left triangle is immediate. The upper right triangle can be checked on simple tensors as follows 
\[
\mu_f \circ(\id_R \otimes \tilde{\tau})(r\otimes t)= \m_f(r\otimes \tau(1\otimes t))=f(r)\tau(1\otimes t).
\]
By commutativity of the triangle diagram for $(T, \tau)$, we know that $f(r)=\tau(r\otimes 1)$. Hence,
\[
f(r)\tau(1\otimes t)=\tau(r\otimes 1) \tau(1\otimes t)=\tau(r\otimes t)
\]
and $\tilde{\tau}$ makes the diagram commute.  Now all that is left to show is that it is unique. If we had some other $\sigma \in \Hom_{\Ring}(T, Z(f))$ such that the diagram commuted, then for the upper right triangle we would need 
\[
\mu_f \circ (\id_R\otimes \sigma)(r\otimes t)=\tau(r\otimes t)
\]
\[
f(r)\sigma(t)=\tau(r\otimes t).
\]
For $r=1$ we see that 
\[
f(1)\sigma(t)=\sigma(t)=\tau(1\otimes t),
\] 
which shows that $\sigma=\tilde{\tau}$. So $\tilde{\tau}$ is a unique map. 
\end{proof}

\begin{corollary}
Let $R \in \mathsf{Ring}$,  (the underlying object of) the centralizer of identity morphism of $R$ is the classical center of the ring $R$.
\end{corollary}
\begin{proof}
By Proposition \ref{203} we know that 
\[
Z(\id_R)=\{r \in R| r \cdot \id_R(r')=\id_R(r')\cdot  r \quad \forall r' \in R\},
\] which is precisely the definition of the center of the ring $R$. The morphism $\mu \colon R\otimes Z(R) \to R$ is simply multiplication which  satisfies the triangle diagram.
\end{proof}

\begin{example}
While the preceding corollary makes it clear that $Z(R)$ is commutative, it is not true that $Z(f)$ must be commutative. Indeed, consider the homomorphism
\[
\varphi \colon \ZZ/2 \to \mathrm{Mat}_{2 \times 2} (\ZZ/2), \quad 1 \mapsto \begin{pmatrix}1&0\\0&1\end{pmatrix}.
\]
The centralizer of $\varphi$, $Z(\varphi)$, is the whole ring of matrices which is not commutative:
\[
\begin{pmatrix}1&1\\0&1\end{pmatrix} \begin{pmatrix}1&0\\1&1\end{pmatrix} = \begin{pmatrix}0&1\\1&1\end{pmatrix} \neq
\begin{pmatrix}1&1\\1&0\end{pmatrix} = \begin{pmatrix}1&0\\1&1\end{pmatrix} \begin{pmatrix}1&1\\0&1\end{pmatrix}.
\]
\end{example}

%%%%%%%%%%%%%%%%%%%%%%%%%%
\section{Functoriality I: Morita Bicategory}\label{func1}

In the 1960's, when B\'{e}nabou  wrote down the definition of bicategories \cite{Benabou}, the \emph{Morita bicategory} was a motivating example.

\begin{definition}
The \emph{Morita bicategory}, $\mathsf{Morita}$, has as objects (unital) rings. For rings $R$ and $S$, $\mathsf{Morita} (R,S)$ is the category of $R$-$S$ bimodules and bimodule homorphisms (intertwiners). Composition in $\mathsf{Morita}$ is given by the tensor product of bimodules. For a ring $R$, the identity object is $R \in \mathsf{Morita}(R,R)$ equipped with its standard $R$-$R$ bimodule structure. The associator and unitors are induced by properties of the tensor product.
\end{definition}

It is a standard exercise that $\mathsf{Morita}$ actually defines a bicategory, i.e., the pentagon and triangle diagrams commute. (Alternatively, see Section XII.7 of \cite{Maclane}.) Moreover, the category $\mathsf{Ring}$ is equipped with a functor to $\mathsf{Morita}$, where a ring homomorphism $\varphi \colon R \to S$ is sent to the $R$-$S$ bimodule $S$ with $R$ acting via $\varphi$. The bicategory is named for Kiiti Morita whose work on equivalences of rings via their categories of modules \cite{Morita}---\emph{Morita equivalence}---can be interpreted as isomorphism in $\mathsf{Morita}$.

\begin{lemma}
Let $f \colon R \to S$ be a ring homomorphism. The centralizer $Z(f)$ has a natural $(Z(S), Z(R))$-bimodule structure.
\end{lemma}
\begin{proof}
	We know that 
	
		\[
		Z(f)=\{ s\in S| sf(r)=f(r)s \mbox{ for any } r\in R\}.
		\]
We can act on some $s'\in Z(f)$ from the left with $s\in Z(S)$ through the multiplication from $S$ as $s\cdot s'=ss'$ where  
	\[
		ss'f(r)=sf(r)s'=f(r)ss'.
	\]
	(Since $s' \in Z(f)$ and $s\in Z(S)$ so  $ss'\in Z(f)$.)
	  We can act on some $s' \in Z(f)$ from the right with some $r\in Z(R)$ through the function $f$ as  
	\[
	s' \cdot r= s' f(r).
	\]
	  For  $r' \in R$
	\[
	s'f(r)f(r')=s'f(rr')=s'f(r'r)=s'f(r')f(r)=f(r')s'f(r)
	\] since $r\in Z(r)$. 
\end{proof}

\begin{lemma}\label{tensorcomp}
Given ring homomorphisms $f \colon R \to S$ and $g \colon S \to T$, the canonical map
\[
\mu \colon Z(g) \otimes_{Z(S)} Z(f) \to Z(g \circ f)
\]
is a map of $(Z(T), Z(R))$-bimodules.
\end{lemma}

\begin{proof}

For some $t\otimes s \in Z(g)\otimes_{Z(s)}Z(f)$ we have that
\[
Z(g)=\{t\in T| g(s)t=tg(s)\mbox{ for any } s\in S\},
\quad Z(f)=\{s\in S| f(t) s=sf(t) \mbox{ for any } t\in T\},
\]
\[
Z(g\circ f)=\{t\in T | g(f(r))t=tg(f(r)) \mbox{ for any } r\in R\}, \; \; \text{ and } \; \; 
\mu(t\otimes s)= t g(s).
\]
First we need to check that this is map is well defined. Let $t\in Z(g)$ so for any $r\in R$ 
\[
g(f(r)) tg(s)=t g(f(r)) g(s)
\]
and since $s\in Z(f)$ and $g$ is a ring homomorphism we know that 
\[
t g(f(r)) g(s)= t g(f(r)s)=tg(sf(r))=t g(s)g(f(r)),
\] 
so $tg(s) \in Z(g\circ f)$.  
By construction $\mu$ is a bimodule homomorphism. 

%
%Next, if for some $t\otimes s$ and $t'\otimes s'$ we have that 
%\[
%\mu(t \otimes s) = \mu(t' \otimes s')
%\]
%or that 
%\[
%tg(s)=t'g(s')
%\]
%then using that we are tensored over $Z(S)$ we know that 
%\ryan{I don't understand the following line!!!}
%\[
%t\otimes s = t g(s) \otimes 1 \quad \mbox{ and } \quad t' \otimes s' = t' g(s') \otimes 1.
%\] 
%If $tg(s)=t'g(s')$, then clearly $t \otimes s = t' \otimes s'$ so this map is indeed one to one.  For any $t\in Z(g\circ f)$ then clearly $t\in Z(g)$ so $\mu(t \otimes 1)=t$ so $\mu$ is onto. 
\end{proof}

\begin{theorem}\label{thm:morita}
The center defines a lax functor $Z: \mathsf{Ring} \to \mathsf{Morita}$.
\end{theorem}

\begin{proof}

For $Z$ defined on on objects by sending a ring $R$ to it's center $Z(R)$. For some $R,S \in \mathsf{Ring}$ for some $f \colon R \to S$ we have the $S$-$R$ bimodule 
\[
Z(f)=\{ s\in S| f(r)s=sf(r) \mbox{ for any } r \in R\},
\] 
so  $Z(f)\in \mathsf{Morita}(Z(R),Z(S))$. As $Z(\id_R) = Z(R)$ as a $Z(R)$-$Z(R)$ bimodule, we define the 2-cell $Z_{\id_R}$ to be the identity, and the (lax) unity axioms hold.

For some $g:S\to T$ Lemma \ref{tensorcomp} shows that we have a well-defined ``compositor" $\mu \colon Z(g) \otimes_{Z(S)} Z(f) \to Z(g \circ f)$. That the map $\mu$ is natural in $g$ and $f$ is trivial in this case as $\mathsf{Ring}$ is an ordinary category. 

Finally, we consider lax associativity. Let $f \colon R \to S$, $g \colon S \to T$, and $h \colon T \to W$ be ring homomorphisms.
Let $(w \otimes t ) \otimes s \in (Z(h) \otimes Z(g) ) \otimes Z(f)$. Tracing this element down the left hand side of the lax associativity diagram in Definition \ref{def:lax} we obtain the element $w h(t) (h \circ g) (s)$, where juxtaposition is the product in the ring $W$. Unwinding and using the fact that $h$ is a ring homomorphism we have that $w h(t) (h \circ g) (s) = w h( t g(s))$ in $W$. The latter expression is obtained by going down the right side of the lax associativity diagram, hence the diagram commutes.
\end{proof}

\begin{example}[Remark 4.13 of \cite{DKR1}]
One might naively expect the map $\mu$ to be an isomorphism, but this is not typically the case. Let $R$ and $T$ be copies of $\RR \oplus \RR$, and $S$ be the ring of $2 \times 2$ upper triangular matrices with entries in $\RR$. Let $f \colon R \to S$ be the inclusion of the diagonal and $g \colon S \to T$ projection onto the diagonal. Then we compute,
\[
Z(R)\cong Z(T) \cong Z(g) \cong Z (g \circ f) \cong \RR \oplus \RR .
\]
Moreover, $Z(S) \cong \RR$, the ring of scalar matrices, and $Z(f) \cong \RR \oplus \RR$, the ring of diagonal matrices. Therefore,
\[
Z(g) \otimes_{Z(S)} Z(f) \cong \left (\RR \oplus \RR \right ) \otimes_\RR \left (\RR \oplus \RR \right ) \cong \RR^{\oplus 4} \not \cong \RR \oplus \RR .
\]

\end{example}

%%%%%%%%%%%%%%%%%%%%%%%%%%%%%%%%%%%
\section{Functoriality II: Cospans}\label{func2}

We now explain another approach to functoriality using cospans. We will see shortly how the construction of the previous section is related to cospans. Cospans are a diagrammatic construction are often easier to interpret in higher categorical settings than purely algebraic notions. For instance, they are commonly used in topological quantum field theory, see \cite{DKR2}.

\begin{definition}
Let $\cC$ be a category which admits small colimits. A \emph{cospan} in $\cC$ is a diagram of the following form.
\begin{center}
\begin{tikzcd}
X \arrow[rd, "f_X"']&  & Y\arrow[ld,"f_Y"]\\
 & W &  
\end{tikzcd}
\end{center}
A \emph{morphism of cospans} is a triple of morphisms $ \phi \colon X \to X'$, $\psi \colon Y \to Y'$, and $\delta \colon W \to W'$ in $\cC$ making the obvious diagram commute.
\end{definition}

A key example for us will be the following. Let $B$ be a $S$-$R$ bimodule which is itself is a ring in a compatible way, i.e., $B$ is a unital $R$-algebra and a unital $S$-algebra and these algebra structures assemble to a $S$-$R$ bimodule structure. Then we have a cospan which records the orbits of the multiplicative unit $1_B \in B$:

\begin{center}
	\begin{tikzcd}
		
		R  \arrow[rd, "\alpha_R"']   &                   &   S\arrow[ld,"\alpha_S"]   \\
		                                                 &       B      &    
		
	\end{tikzcd}
\end{center}
where $\alpha_R (r) = 1_B \cdot r$ and $\alpha_S (s) = s \cdot 1_B$.

\begin{definition}
Let $\cC$ be a category which admits small colimits. For each pair of morphisms $\varphi_X \colon Z \to X$ and $\varphi_Y \colon Z \to Y$ in $\cC$, choose an explicit pushout datum $(P, \iota_X \colon X \to P, \iota_Y \colon Y \to P)$. The \emph{bicategory of cospans in} $\cC$, $\mathsf{coSpan} (\cC)$, has as objects the same objects as $\cC$. For objects $X$ and $Y$, let $\mathsf{coSpan} (\cC) (X,Y)$ be the category of cospans and cospan morphisms with domains $X$ and $Y$, i.e., an object in $\mathsf{coSpan} (\cC) (X,Y)$ is an object $W \in \cC$ and two morphisms $f_X \colon X \to W$ and $f_Y \colon Y \to W$ in $\cC$. The composition bifunctor is given by pushout, i.e., if two cospans glue along a common object then their composition is the (already chosen) pushout along the interior morphisms
\begin{center}
\begin{tikzcd}
X \arrow[rd,"f_X"']& & Y\arrow[ld,"f_Y"']  \arrow[rd, "g_Y"]&  & Z\arrow[ld,"g_Z"]  \\
&  W \arrow[rd, "\iota_W"']&& V\arrow[ld,"\iota_V"] &   \\
 & & P & & 
\end{tikzcd}
\end{center}
now considered as a cospan between $X$ and $Z$. The associator in $\mathsf{coSpan} (\cC)$ is nontrivial as pushouts are only unique up to canonical isomorphism.
\end{definition}

With our definitions recalled we can describe the center as a lax functor, but first we need to verify that it takes values $\mathsf{coSpan} (\mathsf{Ring})$. We will present pushouts in $\mathsf{Ring}$ by amalgamated free products (as $\ZZ$-algebras).
Note that $Z(f)$---for $f \colon R \to S$ a homomorphism---has a ring structure inherited from $S$ (this is a simple application of the subring criteria).

\begin{lemma}
Let $f \colon R \to S$ be a ring homomorphism. The ring structure on $Z(f)$ is compatible with the $Z(S)$-$Z(R)$ bimodule structure.
\end{lemma}
\begin{proof}
As $Z(S)$ is a subring of $S$, it's clear that $Z(f)$ is a $Z(S)$-algebra. Now let $r \in Z(R)$ and $s, s' \in Z(f)$, then
\[
 s(s'\cdot r) = s(s' f(r)) = s(f(r) s') = (s f(r)) s' = (s \cdot r)s',
 \]
as $s'$ commutes with the image of $f$.
\end{proof}

The following has appeared before, specifically as Theorem 4.12 of \cite{DKR1}. We provide a streamlined proof.

\begin{theorem}
The center defines a lax functor $\cZ \colon \mathsf{Ring} \to \mathsf{coSpan} (\mathsf{Ring})$, such that a ring homomorphism $f \colon R \to S$ maps to the cospan
\begin{center} $\cZ(f)  :=$
\begin{tikzcd} Z(R)  \arrow[rd, "\alpha_R"']   &                   &   Z(S) \arrow[ld,"\alpha_S"]   \\
		              &       Z(f)      &    		
\end{tikzcd}.
\end{center}
\end{theorem}

\begin{proof}
As $\cZ(\id_R) = \id_{Z(R)}$, the 2-cell $Z_{\id_R}$ is again the identity.
Next, we define a compositor.  Let $f \colon R \to S$ and $g \colon S \to T$ be ring homomorphisms, we need to compare $\cZ (g) \circ \cZ (f)$ and $\cZ (g \circ f)$, i.e., the cospans
\begin{center}
\begin{tikzcd}[column sep=tiny] Z(R)  \arrow[rd, "\alpha_R"']   & &   Z(S) \arrow[ld,"\alpha_S"'] \arrow[rd,"\beta_S"] && Z(T) \arrow[ld,"\beta_T"]  \\
		              &       Z(f)   \arrow[rd,"\iota_f"']   &    		&Z(g) \arrow[ld,"\iota_g"]\\
		              &&  {\displaystyle Z(g) \coprod_{Z(S)} Z(f)}
\end{tikzcd}
\hspace{0.5in}
\begin{tikzcd}[column sep=tiny] Z(R)  \arrow[rdd, "\gamma_R"']   &    &   Z(T) \arrow[ldd,"\gamma_S"]   \\ \mbox{} \\
		              &       Z(g \circ f)      &    		
\end{tikzcd}.
\end{center}
We need to construct a map of cospans which is natural in $f$ and $g$. We will take the identity morphisms on $Z(R)$ and $Z(T)$, so we need to describe a (natural) ring homomorphism 
\[\varphi \colon Z(g) \coprod_{Z(S)} Z(f) \to Z(g \circ f).
\]
From before, we have the compositor  $\mu \colon  Z(g) \otimes_{Z(S)} Z(f) \to Z(g \circ f)$, so by the universal property of the pushout, it is sufficient to give ring homomorphisms
\[
\varphi_f \colon Z(f) \to Z(g) \otimes_{Z(S)} Z(f) \quad \text{ and } \varphi_g \colon Z(g) \to Z(g) \otimes_{Z(S)} Z(f)
\]
such that for $s \in Z(S)$, $\varphi_f (\alpha_S (s)) = \varphi_g (\beta_S (s))$ and then consider the composition with the map $\mu$. We define
\[
\varphi_f (s) = 1_T \otimes s \quad \text{ and } \quad \varphi_g (t) = t \otimes 1_S.
\]
That the tensor product is over $Z(S)$ exactly implies agreement of $\varphi_f$ and $\varphi_g$ on elements coming from $Z(S)$.

Finally, to verify lax associativity, we can proceed as we did in the proof of Theorem \ref{thm:morita}. That is, consider a reduced word 
\[
\delta^1 \delta^2 \dotsb \delta^k \in \left ( Z(h) \coprod_{Z(T)} Z(g) \right ) \coprod_{Z(S)} Z(f).
\]
Tracing this element down the left side of the lax associativity diagram results in an element of $W$ obtained by (1) applying the map $h$ to any element of $Z(g)$, (2) applying $h \circ g$ to any element of $Z(f)$, and (3) multiplying in $W$. Evaluating our word $\delta^1 \delta^2 \dotsb \delta^k$ along the other direction of the diagram results in an element of the ring $W$ obtained by (1) applying the map $g$ to any element of $Z(f)$, (2) reducing the resulting word (elements in $Z(g \circ f) \subseteq T$ could now be adjacent), (3) applying $h$ to elements of $T$, and (4) multiplying in $W$. These two procedures produce the same element of $W$  as the map $h$ is a ring homomorphism.
 \end{proof}

%
%\begin{remark}
%If it were the case that $Z(f)$ were always a commutative ring, then the preceding theorem would follow directly from Section \ref{func1} as pushouts in commutative rings are balanced tensor products. In this case, the 2-functor defined by the center would not be (op)lax but would actually define a psuedo-functor.
%\end{remark}

%%%%%%%%%%%%%%%%%%%

\section{Centralizers in Topological Quantum Field Theory}\label{tft}

In this last section, we sketch how centers and centralizers arise in topological (quantum) field theory (TFT).  TFTs arise from/model physical systems that are time/scale invariant. Such theories arise at  low energies and are relevant to solid state/condensed matter physics. The recent lectures of Dan Freed, \cite{freed}, are an excellent mathematical introduction to TFTs and the classification of such theories. TFTs are somewhat simpler than the conformal field theories (CFTs) mentioned in the introduction.

We won't give a full definition of topological field theories, but rather focus on the \emph{algebra of observables} of a TFT. Given a physical system (classical or quantum), the quantities one could measure by experiment, the observables, have the mathematical structure of an algebra. Let us restrict to the case of associative $\RR$-algebras, so a vector space over $\RR$ equipped with a compatible ring structure. 

Now let $R$ be an $\RR$-algebra and consider a one dimensional TFT defined on the spacetime $\RR$ whose algebra of observables is $R$. The space $\RR$ occurs as the boundary of the upper half plane $\cH$. One could ask which two dimensional TFTs can be defined on $\cH$ such that they restrict to the given TFT on $\RR$. At the level of observables, the algebra of \emph{boundary observables}, $R$, will be a module over the algebra of \emph{bulk observables}. In turns out there is a universal bulk theory, given by adapting the centralizer construction to TFTs, through which any other bulk theory factors. On observables, the statement is that if $B$ is an algebra of bulk observables, then there is a unique map $B \to Z(R)$ which determines the module structure of $R$ over $B$.

The precise mathematical explanation for the preceding paragraph has quite a rich history going back to a conjecture of Deligne and subsequent work of Kontsevich. The statement on the level of observables is proved in \cite{JT}. One of the key mathematical actors in this play is Hochschild cohomology. The relationship between centers and Hochschild cohomology is described in Section 5.3 of \cite{Lurie}.

\begin{figure}
\begin{tikzpicture}
        \draw[line width=0.4mm, blue] (-2,0)--(2,0);
        \fill [pattern = north east lines, pattern color = black] (-2,0) rectangle (2,2);
        \node at (0,-0.35) {$R$};
        \fill [white] (-0.5,.8) rectangle (0.5, 1.3);
        \node at (0, 1.05) {$Z(R)$};
        \draw[line width=0.4mm, blue] (4,0)--(4,2);
        \draw[line width=0.4mm, red] (4,0)--(8,0);
         \fill [pattern = north east lines, pattern color = black] (4,0) rectangle (8,2);
         \draw[violet,fill=violet] (4,0) circle (.5ex);
         \fill [white] (5.55,.8) rectangle (6.45, 1.3);
         \node at (6,1.05) {$Z(f)$};
         \node at (3.65, 1.05) {$R$};
         \node at (6,-0.35) {$S$};
         \node at (3.65, -0.35) {$S$};
        \end{tikzpicture} 
        \caption{Centers and centralizers as universal algebras of bulk observables associated to boundary topological (quantum) field theories.}
\end{figure}
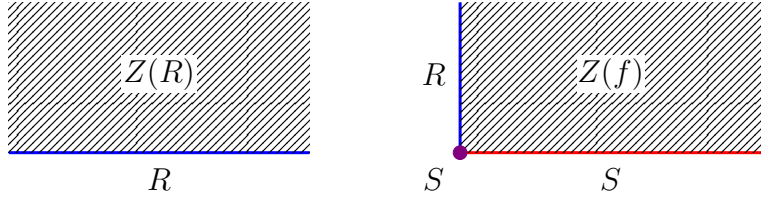

Now, what about the centralizer of a homomorphism $f \colon R \to S$? Consider the boundary of the first quadrant, $\partial \cQ \subset \cQ$, where
\[
\partial \cQ = \{(x,y) \in \RR^2 : x \ge 0 \text{ and } y=0, \text{ or } x=0 \text{ and } y \ge 0\}.
\]
Suppose there is a TFT on $\partial \cQ$ such that the ``vertical" algebra of observables is $R$, the ``horizontal" observables are $S$, and the observables at the origin are given by the $R$-$S$ bimodule $S$ (see the figure above). It turns out that there is a universal bulk theory in this situation as well. At the level of observables, the algebra of bulk observables (on $\cQ$) of the universal theory is given by $Z(f)$.

As the dimension of spacetime increases, the algebras of observables have higher structure beyond that of an associative algebra. The work of Lurie,  \cite{Lurie}, applies to these settings as well. There remain many open questions about the mathematics of TFTs in higher dimensions (see Freed's lectures \cite{freed}).

\vspace{4ex}

\bibliographystyle{plain}
\bibliography{centralizer}

\begin{thebibliography}{10}

\bibitem{Benabou}
Jean B\'{e}nabou.
\newblock Introduction to bicategories.
\newblock In {\em Reports of the {M}idwest {C}ategory {S}eminar}, pages 1--77.
  Springer, Berlin, 1967.

\bibitem{Davydov}
Alexei Davydov.
\newblock Centre of an algebra.
\newblock {\em Adv. Math.}, 225(1):319--348, 2010.

\bibitem{DKR1}
Alexei Davydov, Liang Kong, and Ingo Runkel.
\newblock Field theories with defects and the centre functor.
\newblock In {\em Mathematical foundations of quantum field theory and
  perturbative string theory}, volume~83 of {\em Proc. Sympos. Pure Math.},
  pages 71--128. Amer. Math. Soc., Providence, RI, 2011.

\bibitem{DKR2}
Alexei Davydov, Liang Kong, and Ingo Runkel.
\newblock Functoriality of the center of an algebra.
\newblock {\em Adv. Math.}, 285:811--876, 2015.

\bibitem{freed}
Daniel~S. Freed.
\newblock {\em Lectures on field theory and topology}, volume 133 of {\em CBMS
  Regional Conference Series in Mathematics}.
\newblock American Mathematical Society, Providence, RI, 2019.
\newblock Published for the Conference Board of the Mathematical Sciences.

\bibitem{JY}
Niles Johnson and Donald Yau.
\newblock {\em 2-dimensional categories}.
\newblock Oxford University Press, Oxford, 2021.

\bibitem{Lurie}
J.~Lurie.
\newblock Higher algebra.
\newblock available at
  \href{https://www.math.ias.edu/~lurie/papers/HA.pdf}{Author's Homepage}.

\bibitem{Maclane}
Saunders MacLane.
\newblock {\em Categories for the working mathematician}.
\newblock Graduate Texts in Mathematics, Vol. 5. Springer-Verlag, New
  York-Berlin, 1971.

\bibitem{Morita}
Kiiti Morita.
\newblock Duality for modules and its applications to the theory of rings with
  minimum condition.
\newblock {\em Sci. Rep. Tokyo Kyoiku Daigaku Sect. A}, 6:83--142, 1958.

\bibitem{Riehl}
Emily Riehl.
\newblock {\em Category theory in context}.
\newblock Dover Modern Math Originals. Dover, Mineola, 2016.

\bibitem{JT}
Justin Thomas.
\newblock Kontsevich's {S}wiss cheese conjecture.
\newblock {\em Geom. Topol.}, 20(1):1--48, 2016.

\end{thebibliography}

\end{document}